\documentclass[12pt,amstex]{amsart}

\usepackage{stmaryrd}

\usepackage{epsfig}
\usepackage{amsmath}
\usepackage{amssymb}
\usepackage{amscd}
\usepackage{graphicx}

\usepackage{pstricks}

\usepackage{tocvsec2}

\usepackage{hyperref}

\newtheorem{thm}{Theorem}[section]
\newtheorem{lem}[thm]{Lemma}

\newtheorem{prop}[thm]{Proposition}

\theoremstyle{definition}
\newtheorem{de}[thm]{Definition}

\theoremstyle{remark}

\numberwithin{equation}{section}

\input xy
\xyoption{all}

\def \N {\mathbb N}

\def \Z {\mathbb Z}

\def \B {\mathcal{B}}

\def \X {\mathcal{X}}
\def \Y {\mathcal{Y}}

\def \O {\mathcal{O}}

\def \h {\hat }

\def \a {\alpha }
\def \b {\beta}
\def \ep {\epsilon}

\def \D {\Delta}

\def\w {\omega}

\def \lra{\longrightarrow}

\begin{document}

\title[Weakly mixing, proximal topological models]{Weakly mixing, proximal topological models for ergodic systems and
applications}

\author{Zhengxing Lian}
\author{Song Shao}
\author{Xiangdong Ye}

\address{Wu Wen-Tsun Key Laboratory of Mathematics, USTC, Chinese Academy of Sciences and
Department of Mathematics, University of Science and Technology of China,
Hefei, Anhui, 230026, P.R. China.}

\email{lianzx@mail.ustc.edu.cn}\email{songshao@ustc.edu.cn}
\email{yexd@ustc.edu.cn}

\subjclass[2000]{Primary: 37B05, 37A05} \keywords{topological model, weakly mixing, proximal, minimal point}

\thanks{Authors are supported by NNSF of China (11171320, 11371339)}

\date{April 9, 2014}
\date{May 5, 2014}
\date{May 24, 2014}
\date{May 30, 2014}
\date{July 05, 2014}

\begin{abstract}

In this paper it is shown that every non-periodic ergodic system has two topologically
weakly mixing, fully supported models: one is non-minimal but has a dense set
of minimal points; and the other one is proximal. 
Also for independent interests, for a given Kakutani-Rokhlin tower with relatively prime column heights,
it is demonstrated how to get a new taller Kakutani-Rokhlin tower with same property,
which can be used in Weiss's proof of the Jewett-Krieger's theorem and the proofs of our theorems. Applications of the results are given.


\end{abstract}

\maketitle

\section{Introduction}

A measurable {\em system} is a quadruple $(X,\X, \mu, T)$, where
$(X,\X,\mu )$ is a Lebesgue probability space and $T : X \rightarrow
X$ is an invertible measure preserving transformation. A {\em topological dynamical system} is a pair $(X,
T)$, where $X$ is a compact metric space and $T : X \rightarrow  X$
is a homeomorphism.


\medskip
Let $(X,\X, \mu, T)$ be an ergodic dynamical system. We say that $(\h{X}, \h{\X}, \h{\mu}, \h{T})$ is a {\em topological model}
(or just a {\em model}) for $(X,\X, \mu, T)$ if $(\h{X}, \h{T})$ is a
topological system, $\h{\mu}$ is an invariant Borel probability measure on $\h{X}$ and the systems
$(X,\X, \mu, T)$ and $(\h{X}, \h{\X}, \h{\mu}, \h{T})$ are measure
theoretically isomorphic.

\medskip


\medskip

The theory of topological models is an important part in dynamical systems and has
many applications. The well known
Jewett-Krieger's theorem asserts that every non-periodic ergodic system has a topological model which
is strictly ergodic. Lehrer \cite{Lehrer} showed that we can further require the model to be topologically (strongly)
mixing. We refer to \cite{Glasner, GW06, Weiss89, Weiss00} for surveys and nice results on this topics.
We note that topological models can also be used to obtain the pointwise convergence of non-conventional ergodic averages, \cite{HSY13}.

We mention that the models obtained above are minimal. In this paper we study non-minimal models for a given ergodic system,
and obtain their applications. Here are our main results of this paper. Note that an ergodic system is non-periodic if it has no atom.

\begin{thm}\label{main-result}
\begin{enumerate}
  \item Every non-periodic ergodic system has a topological model which is a non-minimal
  topologically weakly mixing system with a full support and a dense set of minimal points.
  \item Every non-periodic ergodic system has a topological model which is a topologically
  weakly mixing system with a full support and a unique fixed point as its only minimal point.
\end{enumerate}
\end{thm}

Note that a topological system $(X,T)$ with a unique fixed point as its only minimal point is proximal,
i.e. for all $x,y\in X$, $\inf_n d(T^nx,T^ny)=0$. Hence Theorem \ref{main-result}(2) means
that every non-periodic ergodic system has a topological weakly mixing and proximal model with a full support.

\medskip

In Weiss's new proof of the Jewett-Krieger's theorem \cite{Weiss85, Weiss00} and Weiss' theorem on
the doubly minimal model \cite{Weiss95}, a technical complement should be discussed when the column
heights of the Kakutani-Rokhlin tower are not relatively prime. In this paper, we found that one
can avoid this and thus simplify the proofs by using a technical lemma, i.e. Lemma \ref{KR extra prime tower}.
This lemma will be used in the proofs of our theorems and we believe that it will be
useful in other settings.

\medskip
We find two applications of our results. One gives an affirmative answer to a question in \cite{LYY13}
by showing that if $(X,T)$ is a topological system and $(M(X), T_M)$ is the induced system on the probability space,
then the density of minimal points of $(M(X),T_M)$ does not implies $(X,T)$ has the same property. The other
one concerns the existence of a proximal topological K-system which was constructed in \cite{HLY12}. We obtain
a lot of such examples simply using the proximal topological models of any measurable K-systems.

\medskip
\noindent{\bf Acknowledgments:} We would like to thank Wen Huang for very useful suggestions.

\section{Preliminaries}
In this section we recall some notions which we will use in the
following sections.

\subsection{A measurable system}
A measurable system is {\em ergodic} if all
$T$-invariant sets have measures either $0$ or $1$.
For an ergodic system, either the space $X$ consists of a finite set of points on which $\mu$
is equidistributed, or the measure $\mu$ is atom-less. In the first case the system is called
{\em periodic}, and it is called {\em non-periodic} in the latter.

\medskip

A {\em homomorphism} from $(X,\X, \mu, T)$ to a system $(Y,\Y, \nu,
S)$ is a measurable map $\pi : X_0 \rightarrow  Y_0$, where $X_0$ is
a $T$-invariant subset of $X$ and $Y_0$ is an $S$-invariant subset
of $Y$, both of full measure, such that $\pi_*\mu=\mu\circ
\pi^{-1}=\nu$ and $S\circ \pi(x)=\pi\circ T(x)$ for $x\in X_0$. When
we have such a homomorphism we say that the system $(Y,\Y, \nu, S)$
is a {\em factor} of the system $(X,\X, \mu , T)$. If the factor map
$\pi: X_0\rightarrow  Y_0$ can be chosen to be bijective and $\pi^{-1}$ is also
measurable, then we say that the systems $(X,\X, \mu, T)$ and $(Y,\Y, \nu, S)$ are {\em (measure
theoretically) isomorphic}.

\medskip

\subsection{A topological system}

A topological system $(X, T)$ is {\em transitive} if for any non-empty open sets $U,V$ there is some
$n\in \Z_+$ such that $U\cap T^{-n}V\neq \emptyset$. When $X$ has no isolated points, $(X,T)$ is transitive if and only if
there exists some point
$x\in X$ whose orbit $\O(x,T)=\{T^nx: n\in \Z_+\}$ is dense in $X$ and
we call such a point a {\em transitive point}. The system is {\em
minimal} if the orbit of any point is dense in $X$.
A point $x\in X$ is called a {\em minimal point} if
$(\overline{\O(x,T)}, T)$ is minimal. $(X,T)$ is {\em (topologically) weakly mixing}
if the product system $(X\times X, T\times T)$ is transitive.

\medskip

A {\em factor} of a topological system $(X, T)$ is another topological
system $(Y, S)$ such that there exists a continuous and onto map $\phi: X \rightarrow Y$
satisfying $S\circ \phi = \phi\circ T$. In this case, $(X,T)$ is
called an {\em extension } of $(Y,S)$. The map $\phi$ is called a
{\em factor map}.



\subsection{Rokhlin tower}
We need some basic knowledge related to Kakutani-Rokhlin towers. We will use notations
from \cite{Glasner, GW06, Weiss00}. 

\medskip

Let $(X,\X,\mu, T)$ be a dynamical system. Let $B \in \X$. An array
$$\mathfrak{c}= \{B, TB, \ldots, T^{N-1}B \}$$ with $\{T^jB\}_{j=0}^{N-1}$
pairwise disjoint is called a {\em Rokhlin tower} or a {\em column over B
of height $N$}. The set $B$ is called the {\em base} of the tower, and $T^{N-1}B$
is its {\em roof}.  Let $|\mathfrak{c}|=\bigcup_{j=0}^{N-1} T^jB$ the {\em carrier}
of $\mathfrak{c}$. A collection $\mathfrak{t}$ of disjoint columns $\mathfrak{c}_k$
(with bases $B_k$ and heights $N_k$) is called a {\em tower} and let $|\mathfrak{t}|
=\bigcup_{k}|\mathfrak{c}|$. The union of the bases $B=\bigcup_k B_k$ is the
 {\em base} of $\mathfrak{t}$, and the union of the roofs is the {\em roof}
 of $\mathfrak{t}$. The sets $\{T^ix: 0\le i<N_k\}$ for $x\in B_k$ are called the {\em fibers} of $\mathfrak{t}$.

\medskip

Here is the well known Rokhlin's Lemma.

\begin{thm}[Rokhlin's Lemma]
Let $(X,\X,\mu, T)$ be an ergodic system. Given an $\ep>0$ and a natural number $N$,
there exists a Rokhlin tower $\mathfrak{c}$ of height $N$ with base $B\in \X$ such that $\mu(|\mathfrak{c}|)>1-\ep$.
\end{thm}

\subsection{Refining a tower according to a partition}
Let $\mathfrak{t}$ be a tower with columns $\{\mathfrak{c}_k:k\in K\}$
($K$ is finite or countable) and base $B=\bigcup_{k\in K}B_k$. Given a partition
(finite or countable) $\a$, we define an equivalence relation on $B$ as
follows: $x\sim y$ iff $x$ and $y$ are in the same base $B_k$ and for every
$0\le j< N_k$, $T^jx$ and $T^jy$ are in the same elements of $\a$, i.e.
$x$ and $y$ have the same $(\a, N_k)$-name. Now we consider each
equivalence class $B_{k,{\bf a}}$, with {\bf a} an $(\a, N_k)$-name,
as a base of the column $\mathfrak{c}_{k, {\bf a}}=\{B_{k,{\bf a}}, TB_{k,{\bf a}},
\ldots,T^{N_k-1} B_{k,{\bf a}}\}$ and say that the resulting tower $\mathfrak{t}_\a=
\{\mathfrak{c}_{k,{\bf a}}: {\bf a}\in \a^{N_k}, k\in K\}$
is the {\em tower $\mathfrak{t}$ refined according to $\a$}.


\subsection{Kakutani-Rokhlin tower}
For an ergodic system $(X,\X,\mu, T)$, let $B\in \X$ with positive measure,
then it is clear that \(\bigcup_{n\geq 0}T^nB=X\) (mod\ $\mu$). Define the {\em return time function} $r_B: B \rightarrow \N\cup\{\infty\}$ by
\[r_B(x)=\min \ \{n\geq 1:T^n x \in B\}\]
when this minimum is finite and $r_B(x)=\infty$ otherwise. Let $B_k=\{x\in B:r_B(x)=k\}$
and note that by Poincar\'{e}'s recurrence theorem $B_\infty$ is a null set. Let $\mathfrak{c}_k$
be the column $\{B_k,TB_k...,T^{k-1}B_k\}$ and we call the tower
$$\mathfrak{t}=\mathfrak{t}(B)=\{\mathfrak{c}_k:k=1,2...\}$$ the {\em Kakutani tower over $B$}. If the Kakutani tower over B
has finitely many columns (i.e. the function $r_B$ is bounded) we say that $B$ has a {\em finite height}
and we call the Kakutani tower over $B$ a {\em Kakutani-Rokhlin tower} or {\em a K-R tower}.
The number $\max r_B$ is called the {\em height} of B or the {\em height} of K-R tower.


\subsection{Symbolic dynamics}

Let $S$ be a finite alphabet with $m$ symbols, $m \ge 2$. We usually
suppose that $S=\{0,1,\cdots,m-1\}$. Let $\Omega=S^{\Z}$ be the
set of all sequences $x=\ldots x_{-1}x_0x_1 \ldots=(x_i)_{i\in \Z}$, $x_i \in
S$, $i \in \Z$, with the product topology. A metric compatible is
given by $d(x,y)=\frac{1}{1+k}$, where $k=\min \{|n|:x_n \not= y_n
\}$, $x,y \in \Omega$. The shift map $\sigma: \Omega \lra \Omega$
is defined by $(\sigma x)_n = x_{n+1}$ for all $n \in \Z$. The
pair $(\Omega,\sigma)$ is called a {\em shift dynamical system}. Any subsystem of $(\Omega, \sigma)$
is called a {\em subshift system}.
Similarly we can replace $\Z$ by $\Z_+=\{0,1,2,\ldots \}$, and $\sigma$ will be not a
homeomorphism but a surjective map.

Each element of $S^{\ast}= \bigcup_{k \ge 1} S^k$ is called {\em a word} or {\em a block} (over $S$).
We use $|A|=n$ to denote the length of $A$ if $A=a_1\ldots a_n$.
If $\omega=(\cdots \omega_{-1} \omega_0 \omega_1 \cdots) \in \Omega$ and $a \le b \in \Z$, then
$\omega[a,b]=:\omega_{a} \omega_{a+1} \cdots \omega_{b}$ is a $(b-a+1)$-word occurring in
$\omega$ starting at place $a$ and ending at place $b$.
Similarly we define $A[a,b]$ when $A$ is a word. A word $A$ {\em appears} in the word $B$ if there are some $a\le b$ such that
$B[a,b]=A$.

For $n\in\N$ and words $A_1,\ldots, A_n$, we denote by $A_1\ldots A_n$ the concatenation of $A_1,\ldots, A_n$.
When $A_1=\ldots=A_n=A$ denote $A_1\ldots A_n$ by $A^n$. 
If $(X,\sigma)$ is a subshift system, let $[i]=[i]_X=\{x\in X:x(0)=i\}$ for $i\in S$, and
$[A]=[A]_X=\{x\in X:x_0 x_1\cdots x_{(|A|-1)}=A\}$ for any word $A$.


\subsection{Partitions}

Let $(X,\X,\mu, T)$ be a measurable system.
A {\em partition} $\a$ of $X$ is a family of disjoint measurable subsets of $X$ whose union is $X$. Let $\a$ and $\b$ be two partitions of $(X,\X,\mu,T)$. One says that $\a$ {\em refines} $\b$, denoted by $\a\succ \b$ or $\b\prec\a$, if each element of $\b$ is a union of elements of $\a$. $\a \succ \b$ is equivalent to $\sigma(\b)\subseteq \sigma(\a)$, where $\sigma(\mathcal{A})$ is the $\sigma$ algebra generated by the family $\mathcal{A}$.

\medskip

Let $\a$ and $\b$ be two partitions. Their {\em join} is the partition $\a\vee\b=\{A\cap B:A\in \a, B\in \b\}$ and extend this definition naturally to a finite number of partitions. For $m\le n$, define
$$\a_{m}^n=\bigvee_{i=m}^nT^{-i}\a=T^{-m}\a\vee T^{-m+1}\a \vee\ldots\vee T^{-n}\a ,$$
where $T^{-i}\a=\{T^{-i}A: A\in \a\}$.

\subsection{Symbolic representation}
Let $(X,\X,\mu,T)$ be an ergodic system and $\a=\{A_j\}_{1\leq j\leq l}$ a finite
partition (we usually assume $\mu(A_j)>0$ for all $j$). We sometimes think of the partition $\a$
as a function $\xi_0:X\rightarrow \Sigma=\{1,2,\ldots,l\}$ defined by $\xi_0(x)=j$ for $x\in A_j$.
The pair $(X,\a)$ is traditionally called a {\em process}. Let $\Omega=\Omega(l)=\{1,2,\ldots,l\}^\Z$
and let S be the shift. One can define a homomorphism $\phi_\a$ from $X$ to $\Omega$, given by $\phi_\a(x)=\omega\in\Omega$, where
\[\omega_n=\xi_n(x)=\xi_0(T^nx).\]
We denote the distribution of the stochastic process, $(\phi_\a)_*(\mu)$, by $\rho=\rho(X,\a)$ and
call it the {\em symbolic representation measure } of $(X,\a)$. Let
$$X_\a={\rm supp}(\phi_\a)_*\mu={\rm supp} \rho.$$
Then we get a homomorphism $\phi_\a: (X,\X,\mu,T)\rightarrow (X_\a, \X_\a, \rho, S)$.
This homomorphism is called the {\em symbolic representation} of the process $(X,\a)$.
This will not be a model for $(X,\X,\mu,T)$ unless $\bigvee_{i=-\infty}^\infty T^{-i}\a=\X$
modulo null sets, but in any case this does give a model for a non-trivial factor of $X$.

\medskip



\subsection{Copying names}

An important way to produce partitions is by copying or painting names on towers.
If $\mathfrak{c}=\{T^jB\}_{j=0}^{N-1}$ is a column and ${\bf a}\in \Sigma^N$ then
{\em copying the name ${\bf a}$ on the column $\mathfrak{c}$} means that on
$|\mathfrak{c}|=\bigcup_{j=0}^{N-1}T^jB$ we define a partition (may not be on the whole space)  by letting
\begin{equation*}
A_k=\bigcup\{T^jB:{\bf a}_j=k\}, \quad k\in \Sigma=\{1,2,\ldots,l\}.
\end{equation*}
If there is a tower $\mathfrak{t}$ with $q$ columns $\mathfrak{c}_i=\{T^jB_i\}_{j=0}^{N_i-1}$, and
$q$ names ${\bf a}(i)\in \Sigma^{N_i}, i=1,\ldots,q$, then copying these names on $\mathfrak{t}$
means we copy each name ${\bf a}(i)$ on column $\mathfrak{c}_i$, i.e. we define a partition on $|\mathfrak{t}|$ by
\begin{equation*}
   A_k=\bigcup\{T^jB_i: {\bf a}(i)_j=k, i=1,\ldots, q\}, \quad k\in \Sigma.
\end{equation*}
These partitions can be extended to a partition $\a=\{A_1, \ldots, A_l\}$ of the whole space by
 assigning, for example, the value 1 to the rest of the space. Note that we will do this
 in the sequel.

\subsection{A metric on partitions}

For the set of all finite partitions with the same cardinality, there is a complete metric.

\begin{de}
Let $(X,\X,\mu, T)$ be a system. Let $\a=\{A_1,\ldots,A_l\}$ and $\b=\{B_1,\ldots, B_l\}$ be two $l$-set partitions (\(l\ge 2\)), define
\[d(\alpha,\beta)=\mu(\a\D \beta)=\frac 12\sum_{j=1}^l\mu(A_j\Delta B_j).\]
\end{de}

Note that $d(\alpha,\beta)$ will be different when the partitions are indexed in different ways.

\section{An improvement of a technical lemma in Weiss's proof of the Jewett-Krieger' Theorem}\label{section-tech}
In this section we will prove a lemma which is an improvement of a technical lemma
in Weiss's proof of the Jewett-Krieger's Theorem. Using this lemma one may simplify
Weiss's arguments in some sense.


In Weiss's new proof of the Jewett-Krieger' theorem \cite{Weiss85, Weiss00, Glasner} and
in the proof of Weiss' theorem on the doubly minimal model \cite{Weiss95},
one needs the following technical lemma to get a new K-R tower from a given one:

\begin{lem}\cite{Glasner,GW06,Weiss00}\label{KR extra tower}
Let $(X,\B,\mu,T)$ be a non-periodic ergodic system, and let \(\textbf{t}(C_0)\) be a
K-R tower (i.e.  $\max r_{C_0}<\infty$). Then for all $N$ sufficiently large, there exists a set $C_1\subset C_0$ such that
$$N\le r_{C_1}(y)\le N+ 4\max r_{C_0}, \ \forall y\in C_1.$$
That is, the corresponding K-R tower \(\textbf{t}(C_1)\) satisfies ${\rm range}\ r_{C_1}\subset[N,N+4\max r_{C_0}]$.
\end{lem}

When one uses this lemma, one hopes that the column heights of the K-R tower are relatively prime, which is not
guaranteed in Lemma \ref{KR extra tower}. Hence in Weiss's new proof of the Jewett-Krieger's theorem, one
first assumes that the system has no rational spectrum, in which case automatically the column heights
of every K-R tower are relatively prime. Then one deals with other cases.
The following lemma will avoid this kind of discussion. 

\begin{lem}\label{KR extra prime tower}
Let $(X,\B,\mu,T)$ be a non-periodic ergodic system. Let $\mathfrak{t}'$ be a K-R tower with
bases $C_i$ and heights $h_i$, $1\le i\le k$, and let $N=\max_i\{h_i\}$ and $C=\bigcup_{i=1}^kC_i$.
Assume that $h_1,h_2,\ldots,h_k$ are relatively prime. Then for any $n$ large enough,
there is a K-R tower $\mathfrak{t}$ with base $D$ such that:
\begin{enumerate}
     \item $D\subset C$;
     \item $r_D(y)\in [n,n+6N], \forall y\in D$;
     \item the column heights of $\mathfrak{t}(D)$ are relatively prime.
\end{enumerate}
\end{lem}

\begin{proof}
First we will find a set $\hat{D}\subset C$ with the following two properties:

(i) $n+N\le r_{\hat{D}}(y)\le n+5N, \forall y\in \hat{D}$ and

(ii) $\mu(\hat{D}\cap C_i)>0$ for each $1\leq i\leq k$.

Then according to the second property of $\h{D}$, we adjust some part of $\h{D}$ to get $D$
such that the column heights of $\mathfrak{t}(D)$ are relatively prime.

\medskip

\noindent {\bf Step 1: The construction of $\hat{D}$.} Now we describe how to get $\h{D}$.
To that aim, we first construct a Kakutani tower
$\mathfrak{t}(\h{B})$ with height larger than $10(n+3N)^2$ and $\mu(\hat{B}\cap C_i)>0$
for each $i$. But at this point we may have $\max r_{\h{B}}=\infty$ (i.e. $\mathfrak{t}(\h{B})$
may not be a K-R tower). So we need to modify it such that the resulting tower is a K-R tower $\mathfrak{t}(\h{D})$.

\medskip

By Rokhlin Lemma, there is a $B\subset C$ such that the Rokhlin tower
$$\mathfrak{c} = \{B, TB, \ldots, T^{M-1}B \}$$
 satisfies that $M>20(n+3N)^2$ and $\mu(B)<\frac{\min_i\{\mu(C_i)\}}{10k(n+3N)^2+k}$.

Let $n_0=0$. Now find the smallest $n_1\in \N$ with
\begin{description}
  \item[($a_1$)] $n_1-n_0\geq 10(n+3N)^2$;
  \item[($b_1$)] $\mu(T^{n_1}B\cap (\bigcup_{j=1}^k C_j\setminus (\bigcup_{j=0}^{10(n+3N)^2}T^jB)))>0$.
\end{description}
Hence there is some $d_1\in \{1,2,\ldots,k\}$ such that
$$\mu(T^{n_1}B\cap( C_{d_1}\setminus (\cup_{j=0}^{10(n+3N)^2}T^jB)))>0.$$
Let $$B_1=T^{n_1}B\cap( C_{d_1}\setminus (\cup_{j=0}^{10(n+3N)^2}T^jB)).$$
Inductively, assume that for $1\le i\le k-1$ we have obtained $n_1,\ldots n_{i}$,
distinct numbers $d_1,\ldots,d_i\in \{1,\ldots,k\}$ and measurable sets $B_1, \ldots, B_i$.

Let $n_{i+1}$ be the smallest natural number satisfying:
\begin{description}
  \item[($a_{i+1}$)] $n_{i+1}-n_i\geq 10(n+3N)^2$;
  \item[($b_{i+1}$)] $\mu(T^{n_{i+1}}B\cap (\bigcup_{j=1}^k C_j\setminus (\bigcup_{j=1}^i
  C_{d_i}\cup\bigcup_{s=0}^i \bigcup_{j=n_s}^{n_s+10(n+3N)^2}T^jB)))>0$.
\end{description}
Hence there is some $d_{i+1}\in \{1,2,\ldots,k\}\setminus \{d_1,\ldots,d_i\}$ such that
$$\mu(T^{n_{i+1}}B\cap (C_{d_{i+1}}\setminus (\bigcup_{j=1}^i C_{d_i}\cup \bigcup_{s=0}^i
\bigcup_{j=n_s}^{n_s+10(n+3N)^2}T^jB)))>0.$$ Let
$$B_{i+1}=T^{n_{i+1}}B\cap (C_{d_{i+1}}\setminus (\bigcup_{j=1}^i C_{d_i}\cup\bigcup_{s=0}^i
\bigcup_{j=n_s}^{n_s+10(n+3N)^2}T^jB)).$$ Note that $B_{i+1}=T^{n_{i+1}}B\cap (C_{d_{i+1}}\setminus
(\bigcup_{s=0}^i \bigcup_{j=n_s}^{n_s+10(n+3N)^2}T^jB))$. This inductive process can be done
for $i=2,3,\ldots,k$ since $\mu(B)<\frac{\min_i\{\mu(C_i)\}}{10k(n+3N)^2+k}$, which means
$\mu(\bigcup_{i=0}^{k-1}\bigcup_{j=0}^{10(n+3N)^2}T^{n_i+j}B)<\mu(C_s)$, $1\leq s\leq k$.

\medskip

Now by induction we obtain subsets $B_1,\ldots,B_k$. Let $\hat{B}=\bigcup_{i=1}^kB_i$. We claim that:

\medskip
\noindent {\bf the height of each column in the Kakutani tower $\mathfrak{t}(\hat{B})$ is larger than
$10(n+3N)^2$, i.e. $r_{\h{B}}(y)\ge 10(n+3N)^2, \forall y\in \h{B}$.}

\medskip

To prove the claim, we need to prove that for any $l>0$ and $1\leq u,v \leq k$,
$\mu(T^lB_u\cap B_v)>0 $ implies $l\geq 10(n+3N)^2$. Since $\mu(T^lB_u\cap B_v)>0 $,
 there is a subset $P\subset T^lB_u\cap B_v$ with positive measure. If $u=v$,
 then $l\geq 10(n+3N)^2$ since $B_u\subset B$. If $u<v$, then $l\geq 10(n+3N)^2$
 since $T^{-l}P\subset B_u\subset B,\;P\subset B_v$ and $\mu((\bigcup_{j=n_u}^{n_u+10(n+3N)^2}T^jB)
 \cap B_v)=0$. Finally assume $u>v$. Since $n_u$ is the first number satisfing the inductive
 condition ($a_u$), we have $\mu(\bigcup_{j=n_{v-1}+10(n+3N)^2+1}^{n_v-1}T^j B\cap B_u)=0$.
 We also have $\mu(\bigcup_{n_{v-1}}^{n_{v-1}+10(n+3N)^2}B\cap B_u)=0$, so
 $\mu((\bigcup_{j=n_{v-1}}^{n_v-1}T^jB)\cap B_u)=0$. Since $n_u-n_{u-1}\geq 10(n+3N)^2$
 and $T^{-l}P\subset B_u\cap T^{-l}B_v$, we conclude that $l\geq 10(n+3N)^2$.

\medskip

By the construction we also see that $\mu(\hat{B}\cap C_i)>0$ for each $i\in \{1,\ldots, k\}$.
Since $n+3N$ and $n+3N+1$ are relatively prime, we may partition each column of $\mathfrak{t}(\hat{B})$
into blocks of sizes $n+3N$ and $n+3N+1$. And then we move the base level of each block to the nearest
level that belongs to $C$. Collect the union of the base level and $\hat{B}$, and we get a set $\hat{D}\subset C$ satisfying

\medskip

\begin{description}
  \item[$(I)$]  The height of $\mathfrak{t}(\hat{D})$ ranges in $[n+N,n+5N]$;
  \item[$(II)$]  $\mu(\hat{D}\cap C_i)>0$ for each $1\leq i\leq k$.
\end{description}
The set of heights of $\mathfrak{t}(\h{D})$ may not be relatively prime, and we need modify it to what we need.
\medskip

\noindent {\bf Step 2: The construction of ${D}$.} For each $i\in \{1,\ldots,k\}$,
let $E_i\subset \hat{D}\cap C_i$ be a measurable subset with positive measure.
Then we get $k$ sets $E_1,E_2,\ldots,E_k$ with the
corresponding heights $\hat{h}_1,\hat{h}_2,\ldots,\hat{h}_k$ respectively. Since $T$ is
non-periodic and ergodic, $\mu(E_i\setminus T^{\hat{h}_i}E_i)>0$ for each $i$.
Let $\ep_0=\min_i\{\mu(E_i\setminus T^{\hat{h}_i}E_i)\}$.

Let $F_1\subset E_1\setminus T^{\hat{h}_1}E_1$ be a subset satisfying
$0<\mu(F_1)<\frac{\ep_0}{2^{k+1}}$. Inductively assume
for $1\le i\le k-1$ we have constructed subsets $F_1,\ldots,F_i$ satisfying
$2\mu(F_j)\leq\mu(F_{j+1})<\frac{\ep_0}{2^{k-j+1}}$ for each $1\le j\le i-1$.
Note that
\begin{equation}\label{sum}
    \sum_{j=1}^i\mu(F_j)\le \sum_{j=1}^i\frac{\ep_0}{2^{k-j+2}}=\frac{\ep_0}{2^{k-i+2}}\frac{1-1/2^i}{1-1/2}<\frac{\ep_0}{2^{k-i+1}}.
\end{equation}
Thus $\mu((E_{i+1}\setminus T^{\hat{h}_{i+1}}E_{i+1})\setminus (\bigcup_{j=1}^i T^{\hat{h}_j}F_j ))>\ep_0(1-\frac{1}{2^{k-i+1}})$.
Hence one can find
$$F_{i+1}\subset (E_{i+1}\setminus T^{\hat{h}_{i+1}}E_{i+1})\setminus (\bigcup_{j=1}^i T^{\hat{h}_j}F_j )$$ satisfying $2\mu(F_i)\leq\mu(F_{i+1})<\frac{\ep_0}{2^{k-i+1}}$.

\medskip

In such a way by induction we get $k$ sets $F_1,F_2,\ldots,F_k$.
For each $i$, we have the following properties:

\medskip

\begin{description}
 \item[($i$)] $F_i \subset C_i$, which implies $T^{h_i}F_i \subset C$;
 \item[($ii$)] $F_i \subset E_i\setminus T^{\hat{h}_i}E_i$, which implies $T^{\hat{h}_i} F_i \subset \hat{D}\setminus F_i$;
 \item[($iii$)] For $j\ge i$, $T^{\hat{h}_i}F_i\cap F_j=\varnothing$ and $\mu(T^{\hat{h}_i} F_i)=\mu(F_i)>\Sigma_{s=1}^{i-1} \mu(F_s)$.
 \item[($iv$)] $\mu(T^{\hat{h}_i}F_i \cap (\hat{D}\setminus (\bigcup_{j=1}^k F_i)))>0$.
 \end{description}

Note (i), (ii) and the first part of (iii) follow from the definition of $F_i$. The second part of (iii)
follows from the inequality $\mu(F_{j+1})\ge 2\mu(F_j)$, i.e.
$$\mu(T^{\hat{h}_i} F_i)=\mu(F_i)\ge 2\mu(F_{i-1})\ge \mu(F_{i-1})+2\mu(F_{i-2})\ge \ldots >\Sigma_{s=1}^{i-1} \mu(F_s). $$
And (iv) is deduced from (iii) readily.

\medskip

Finally we put $D=(\hat{D}\setminus(\bigcup_{i=1}^k F_i ))\cup (\bigcup_{i=1}^k T^{{h}_i} F_i)$.
By the properties of $\{F_i\}$, we conclude:

\medskip

 \begin{enumerate}
   \item $D\subset C$.
   \item $\mathfrak{t}(D)$ is a K-R tower, and the height of $\mathfrak{t}(D)$ ranges in $[n,n+6N]$.
   \item The collection of the column heights of the K-R tower $\mathfrak{t}(D)$ contains
   $\{\hat{h}_i,\hat{h}_i-h_i\}_{i=1}^k$, which are relatively prime since $\{h_i\}_{i=1}^k$ are relatively prime.
 \end{enumerate}

(1) is followed by the definition of $D$, and (2) is from $(i)$ above. By $(iv)$ and the
definition of $D$ for each $i\in \{1,\ldots,k\}$ there is some column of $\mathfrak{t}(D)$
with height $\hat{h}_i-h_i$. By (\ref{sum}), we have that $\frac{\mu(E_i)}{2}\ge \frac{\ep_0}{2}>\Sigma_{i=1}^k \mu(F_i)$,
which implies that for each $i\in \{1,\ldots,k\}$ there is some column of $\mathfrak{t}(D)$ with height $\hat{h}_i$.
Hence we have (3).

\medskip

The tower $\mathfrak{t}(D)$ is as required. The proof is completed.
\end{proof}

\section{Proof of Theorem \ref{main-result}-(1)}

A subset $S$ of $\Z_+$ is {\it syndetic} if it has a bounded gaps,
i.e. there is $N\in \N$ such that $\{i,i+1,\cdots,i+N\} \cap S \neq
\emptyset$ for every $i \in {\Z}_{+}$. $S$ is {\it thick} if it
contains arbitrarily long runs of positive integers, i.e. there is a
strictly increasing subsequence $\{n_i\}$ of $\Z_+$ such that
$S\supset \bigcup_{i=1}^\infty \{n_i, n_i+1, \ldots, n_i+i\}$.
Some dynamical properties can be interrupted by using the notions of
syndetic or thick subsets. For example, a classic result of
Gottschalk and Hedlund \cite{GH} stated that $x$ is a minimal point if and only if
$$N(x,U)=\{n\in\Z_+: T^nx\in U\}$$ is syndetic for any neighborhood $U$ of $x$,
and by Furstenberg \cite{F67} a topological system
$(X,T)$ is weakly mixing if and only if
$$N(U,V)=\{n\in\Z_+: U\cap T^{-n}V\neq \emptyset \}$$ is thick for any
non-empty open subsets $U,V$ of $X$.

\medskip

A set $S$ is called {\em thickly syndetic} if for every $N$ the positions where length $N$ runs begin form a syndetic set.
A subset $S$ of $\Z_+$ is {\it piecewise syndetic} if it is an
intersection of a syndetic set with a thick set. It is known that a topological system
$(X,T)$ is an $M$-{\it system} (i.e. the set of minimal point of $(X,T)$ is dense) if and only if there is a transitive point $x$ such that $N(x,U)$ is piecewise syndetic for any neighborhood $U$ of
$x$ (see for example \cite[Lemma 2.1]{HY}). We will use this fact in the sequel.

To prove Theorem \ref{main-result}-(1), we begin with the following observation.


\begin{lem}\label{infinite}
Let $(X,\X,\mu,T)$ be a non-periodic ergodic system. Then there is a tower whose set of the column heights is infinite.
\end{lem}

\begin{proof}
Given a tower with base $C_1$, if the set of column heights is infinite then we are done; or
we put it to be $\{h_1,\ldots,h_{n_1}\}$. Let $C_1^i$ be the corresponding column-base
with the height $h_i$, and we may assume that $h_1<\ldots < h_{n_1}$ (by putting the column-bases with the same
height together to form a new column-base). Choose a measurable set $E_1\subset T^{h_{n_1}}C_1^{n_1} $ such
that $0< \mu(E_1) < \frac 12\min_{1\leq i\leq n_1}\{\mu(C_1^i)\}$.

Let $C_2 = C_1\setminus E_1$ and we have a tower with base $C_2$. If the set of the column heights is
infinite then we are done. Or we have a bigger height set than the tower with base $C_1$, and
let it be $\{h_1,\ldots,h_{n_1},h_{n_1+1}\ldots,h_{n_2}\}$. Let $C_2^i$ be the corresponding
column-base with the height $h_i$, and we assume that $h_1<\ldots < h_{n_1}<\ldots < h_{n_2}$.
Choose a measurable set $E_2\subset T^{h_{n_2}}C_2^{n_2} $ such that $0< \mu(E_2) < \frac {1}{2^2}\min_{1\leq i\leq n_2}\{\mu(C_2^i)\}$.

Let $C_3 = C_2\setminus E_2$ and continue the process above. If after finite steps we get a tower
with infinitely many heights, then we are done. Or we will have a sequence of towers with deceasing
bases $\{C_k\}$, $n_1<\ldots<n_k$ and measurable sets $E_j$ with $0< \mu(E_j) < \frac {1}{2^j}\min_{1\leq i\leq n_j}\{\mu(C_j^i)\}$
for $1\le j\le k$ such that
$$\mu(C_{k+1})\ge \mu(C_{k})-\mu(E_k)>(1-\frac{1}{2^k})\mu(C_k)$$ for all $k\in \N.$

Let
$$C=\bigcap_{k=1}^\infty C_k$$
Then $\mu(C)>0$ and the tower with base $C$ has infinitely many heights. The proof is completed.
\end{proof}

We follow the standard procedure to prove Theorem \ref{main-result}-(1). Namely, first for a given partition
$\hat{\a}$ we construct a partition $\alpha$ close to $\hat{\alpha}$ such that the corresponding symbolic representation
$(X_\a,\X_\a,\rho,S)$ is a non-minimal topologically weakly mixing system with a dense set of minimal points. Then
we use the inverse limit by a more delicate argument. Finally we show the resulting system is the one which we need.

\begin{prop}\label{weak-nonminimal-dense minimal}
Let $(X,\X,\mu,T)$ be a non-periodic ergodic system and let $\hat{\a}$ be a finite partition of $X$. For each $\ep>0$, there is a partition $\a$ such that the corresponding symbolic representation $(X_\a,\X_\a,\rho,S)$ is a non-minimal topologically weakly mixing system with a dense set of minimal points, and
$$d(\a,\hat{\a})<\ep.$$
\end{prop}

\begin{proof}
By Lemma \ref{infinite}, there is a tower consisting of infinitely many columns with different heights.
Precisely, let $\mathfrak{t}(C)$ be a tower as in Lemma \ref{infinite} with columns $\{\mathfrak{c}_k:k\in \N\}$
and base $C=\bigcup_{k\in \N}C_k$.
Let $h_k$ be the height of column $\mathfrak{c}_k$, and assume that $h_1<h_2<\ldots$.
Note that for $k$ large enough $|\mathfrak{c}_k|$ will be very small. We will adjust some $|\mathfrak{c}_k|$ to get what we need.

Let $\a_0=\hat{\a}=\{\hat{A}_1, \ldots, \hat{A}_k\}$. For each $m\in \N$, let $\w_m = v_1v_2\ldots v_{k^m}$, where $$v_{_{a_1k^{m-1}+a_2k^{m-2}+\ldots+a_m}}=(a_1,a_2,\ldots,a_m),$$ for each $(a_1,\ldots,a_m)\in \{1,2,\ldots,k\}^m$.
That is, each $v_i$ is a word of length $m$ and $\w_m$ is a word which contains all
the $m$-name in $\{1,2,\ldots,k\}^m$. Note that $|\w_m|=mk^m$. 

\medskip

Before going on, let us recall the notion of copying a name on the column.
Let $\mathfrak{c}=\{T^jB\}_{j=0}^{h-1}$ be a column and ${\bf a}\in \{1,\ldots,k\}^N$ with $N\le h$. Then  copying the name ${\bf a}$ on the column $\mathfrak{c}$ means that we copy the name ${\bf a}$ on the first $N$ levels of $\mathfrak{c}$. That means, for the new partition $\{A_1,\ldots,A_k\}$ one has that
$$T^{i-1}B\subset {A}_{a_i},\;\;1\leq i \leq N ,$$ where ${\bf a}=(a_1,\ldots,a_{N})\in \{1,\ldots,k\}^{N}$.

\medskip
\noindent \textbf{Step 1:}
Since $\sum_k |\mathfrak{c}_k|<\infty$, there are columns $\mathfrak{c}_{n_1^1}, \mathfrak{c}_{n_1^2}$ such that $\mu(|\mathfrak{c}_{n_1^1}\cup\mathfrak{c}_{n_1^2}|)<\frac{\ep}{2}$ and $h_{n^2_1}>h_{n^1_1}>2 k^2$.
Let $\xi_{1_1^1}=\w_11^{h_{n^1_1}-k}\in \{1,\ldots,k\}^{h_{n_1^1}}$, where $1^{j}=(1,1,\ldots,1)$ with the length $j$. And let $\xi_{1_1^2}=\w_21^{h_{n^2_1}-k}\in \{1,\ldots,k\}^{h_{n_1^2}}$. For $i=1,2$, copy the name $\xi_{1_1^i}$ to the column $\mathfrak{c}_{n_1^i}$, and we get a partition $\a_1$. Note that $d(\a_0,\a_1)=d(\hat{\a},\a_1)<\frac{\ep}{2}$. The first step of adjustment is finished.

\medskip
\noindent \textbf{Step $m$:}
In general, for each $m\in \N$, choose columns $\mathfrak{c}_{n_m^1},\ldots, \mathfrak{c}_{n_m^{m+1}}$ such that $$\mu(\bigcup_{i=1}^{m+1}|\mathfrak{c}_{n_m^i}|)<\frac{\ep}{2^m},$$ and assume that $h_{n_m^{m+1}}
>\ldots >h_{n_m^1}>2m^2k^{2m}$. For $1\leq i\leq m$, let
$$\xi_{n_m^i}=\underbrace{\w_{2i-1}\w_{2i-1}\ldots\w_{2i-1} }_{m\ \text{times}}1^{h_{n_m^i}-m(2i-1)k^{2i-1}}=(\w_{2i-1})^i1^{h_{n_m^i}-m(2i-1)k^{2i-1}}\in \{1,\ldots,k\}^{h_{n_m^i}}.$$
And let
$$\xi_{n_m^{m+1}}=\w_{2m}1^{h_{n_m^{m+1}}-2mk^{2m}}\in \{1,\ldots,k\}^{h_{n_m^{m+1}}}.$$
Now for $1\leq i \leq m+1$, copy the name $\xi_{n_m^i}$ to the column $\mathfrak{c}_{n_m^i}$,
and we get a new partition $\a_m$. Note that $d(\a_{m-1},\a_m)<\frac{\ep}{2^m}$.

Moreover, note that we do
copying on the $m$ columns $\mathfrak{c}_{n_m^1},\ldots, \mathfrak{c}_{n_m^{m}}$ to make sure that the
symbolic representation of the resulting partition
is not minima but has a dense set of minimal points, and we do copying on $\mathfrak{c}_{n_m^{m+1}}$ to make sure
the symbolic representation of the resulting partition is weakly mixing. Of course we can do it on a single column,
but this will cause complication when dealing with the situation in Proposition \ref{increasing weak-nonminimal-dense minimal}.

\medskip

For all $m\in \N$, we make the above adjustment, and we obtain a new partition $\a=\{A_1, \ldots, A_k\}$.
It is clear that $$d(\hat{\a},\a)\le \sum_{i=1}^\infty d(\a_{i-1},\a_i)<\sum_{i=1}^{\infty}\frac{\ep}{2^i}=\ep.$$

\noindent \textbf{Properties of $\alpha$:}
Now we prove that $(X_\a,\X_\a,\rho,S)$ is non-minimal, weakly mixing, and the set of minimal points is dense.

\medskip

To show $(X_\a, S)$ is weakly mixing, it suffices to show that for each $m\in\N$,
$E_1,F_1,E_2,F_2\in\bigvee_{i=0}^{m-1}T^{-i}\a$ with positive measures, the following holds 
$$\mu\times \mu(E_1\times F_1 \cap (T\times T)^{-m} E_2\times F_2)>0.$$
This depends on the adjustment on $\mathfrak{c}_{n_m^{m+1}}$. 

Denote $\w_{2m}$ by $\w_{2m}=u_1u_2\ldots u_{k^{2m}}$, where $\{u_j\}_{j=1}^{k^{2m}}=\{1,2,\ldots, k\}^{2m}$.
Let the names of $E_1,E_2,F_1$ and $F_2$ be $e_1, e_2,
f_1, f_2\in \{1,\ldots,k\}^m$ respectively.  Then $e_1e_2=u_s$ and $f_1f_2=u_t$ for some $1\le t,s\le k^{2m}$. By the construction of $\a$, it follows that
$$T^{2m(s-1)}C_{n_m^{m+1}}\subset E_1\cap T^{-m}E_2, \ \text{and} \ T^{2m(t-1)}C_{n_m^{m+1}}\subset F_1\cap T^{-m}F_2 .$$ Thus,
$$(E_1\times F_1)\cap (T\times T)^{-m}(E_2\times F_2)\supset T^{2m(s-1)}C_{n_m^{m+1}}\times T^{2m(t-1)}C_{n_m^{m+1}}.$$
In particular, we conclude that
$$\mu\times \mu((E_1\times F_1)\cap (T\times T)^{-m}(E_2\times F_2))\ge \mu\times \mu ( T^{2m(s-1)}C_{n_m^{m+1}}\times T^{2m(t-1)}C_{n_m^{m+1}})>0.$$

\medskip

To see $(X_\a, S)$ is a non-minimal $M$-system, we show that each transitive point $w\in X_\a$ is
piecewise syndetically  but not syndetically recurrent. Let $x\in X$ such that $\phi_\a(x)=w$. It is easy to see that $w\neq 1^\infty$.

Let $w=(\ldots, a_{-2},a_{-1},a_0,a_1,a_2,\ldots)$. Then for each $m\in \N$,
$$[w]_{-m+1}^{m-1}=\{p\in X_\a: p[{-m+1},{m-1}]=(a_{-m+1},\ldots,a_0,a_1,\ldots,a_{m-1})\}$$ is a neighborhood of $w$.
Let $A\in\bigvee_{i=-m+1}^{m-1}T^{-i}\a$ with the name $(a_{-m+1},\ldots,a_{m-1})$. Since $w\neq 1^\infty$,
it is clear that when $m$ large enough we have $(a_{-m+1},\ldots,a_{m-1})\neq 1^{2m-1}$.

As defined before, $\w_{2m-1} = v_1v_2\ldots v_{k^{2m-1}}$, where $\{v_i\}_{i=1}^{k^{2m-1}}=\{1,\ldots,k\}^{2m-1}$.
Then  $(a_{-m+1},\ldots,a_{m-1})=v_r$ for some $r$. For each $j\geq m$, by the definition
of $x$, one can find $l_j$ such that $T^{l_j}x\in C_{n_j^m}$. By the construction of $\a$, for
$1\leq i\leq j$, $T^{i(2m-1)k^{2m-1}+(2m-1)(r-1)+m}C_{n_j^m}\subset A$. That means, for each $j\ge m$,
$$\{l_j+i(2m-1)k^{2m-1}+(2m-1)(r-1)+m\}_{1\leq i\leq j}\subset N(w,[w]_{-m+1}^{m-1}),$$ which implies
 $N(w, [w]_{-m+1}^{m-1})$ is piecewise syndetic.

 On the other hand, for each $j>m$ and
 $m-1<i<h_{n_j^m}-j(2m-1)k^{2m-1}$, $T^{l_j+j(2m-1)k^{2m-1}+i}x\in \bigcap_{d=-m+1}^{m-1}T^{-d} A_1$.
 As $(a_{-m+1},\ldots,a_{m-1})\neq 1^{m}$, we have that $\bigcap_{d=0}^{m-1}T^{-d} A_1\cap A=\emptyset$
 and hence $T^{l_j+j(2m-1)k^{2m-1}+i}x\not \in A$, which implies for each
 $j>m$,
 $$\{l_j+j(2m-1)k^{2m-1}+i\}_{m-1<i<h_{n_j^m}-j(2m-1)k^{2m-1}}\cap N(w, [w]_{-m+1}^{m-1})=\emptyset.$$
 Since $h_{n_j^m}-j(2m-1)k^{2m-1}\geq 2j^2k^{2j}-j(2m-1)k^{2m-1}$ and
 $$2j^2k^{2j}-j(2m-1)k^{2m-1}\stackrel {j\to \infty }{\longrightarrow }\infty,$$
 we conclude that $N(w, [w]_{-m+1}^{m-1})$ is not syndetic. The proof is completed.
\end{proof}

\medskip

To prove Theorem \ref{main-result}-(1), Proposition
\ref{weak-nonminimal-dense minimal} is not enough. In fact we need to get an increasing sequence of
required partitions $\gamma_n$ such that the inverse limit of the corresponding symbolic
representations is what we need. The following simple fact will used.

\begin{lem}\label{Fact}
Let $\a=\{A_1,\ldots, A_a\}, \a'=\{A'_1,\ldots, A'_a\},$ and $\b=\{B_1,\ldots,B_b\}$ be partitions
with $\a \succ \b$. Then there is a natural way to get a partition $\b'$ such that $\a'\succ \b'$.
Moreover, if $d(\a,\a')<\ep$, then we also have $d(\b,\b')<\ep$.
\end{lem}

To see it we note that $\a \succ \b$ defines a function
$$\phi:\{1,\ldots,b\}\rightarrow 2^{\{1,\ldots,a\}}\setminus \emptyset$$
such that $A_{x}\subset B_{y}$ iff  $x\in \phi(y)$.
Let
$$\b'=\{B'_1,\ldots,B'_b\},\ B'_s=\bigcup_{t\in \phi(s)}A'_t .$$
Notice that if $d(\a,\a')<\ep$, then we also have $d(\b,\b')<\ep$, since it is easy to check that
$(A_1\cup B_1)\Delta (A_2\cup B_2)\subset A_1\Delta A_2\cup B_1\Delta B_2$.


\begin{prop}\label{increasing weak-nonminimal-dense minimal}
Let $(X,\X,\mu,T)$ be a non-periodic ergodic system. Then there exists an increasing sequence of finite
partitions $\{\gamma_n\}$ such that $\sigma(\gamma_n)\nearrow \X$ and for each $n\in \N$ the corresponding
symbolic representation $(X_{\gamma_n},\X_{\gamma_n},\rho_n,S)$ is a non-minimal topologically
weakly mixing system with a dense set of minimal points.
\end{prop}

\begin{proof}
The basic idea of the proof is the same as in the proof of of Proposition \ref{weak-nonminimal-dense minimal}.
Since we have to deal with countably many partitions, we need to do some small modifications with the proof.

\medskip

Let $\{\b_n\}$ be an increasing sequence of finite partitions such that $\sigma(\b_n)\nearrow \X$. First
we fix the same tower $\mathfrak{t}(C)$ as in the proof of Proposition \ref{weak-nonminimal-dense minimal}
and let $\{\ep_n\}$ be a sequence with $\sum_{n=1}^{\infty} \ep_n<\infty$.

For $\b_1$, we adjust $\a_1$ as in  Step 1 of the proof of Proposition \ref{weak-nonminimal-dense minimal}
to get a new partition $\gamma_1^1$. We replace $\b_2$ by $\b_2\bigvee \gamma_1^1$, and thus we have $\gamma_1^1\prec\b_2$.
Then continue our induction. To be precise, we rewrite the Step $m$.

\medskip

\textbf{Step $m'$:}
We replace $\b_m$ by $\b_m\bigvee \gamma_{m-1}^{m-1}$ (still denote it by $\b_m$), and thus we have
$\gamma_{m-1}^{m-1}\prec \b_m$. Let $\b_j=\{B^j_1, B^j_2, \ldots, B^j_{k_j}\}$ for $1\le j\le m$.
Since $\b_1\prec \ldots \prec \b_m$, we may assume that $B^j_i\subset B^{j-1}_i$, for each $2\le j\le m$, $1\leq i\leq k_{j-1}$.

Recall that the word $\w_m$ depends on the cardinality of the partition in the proof of Proposition
\ref{weak-nonminimal-dense minimal}. Unlike the situation there, 
now the cardinalities of partitions are increasing. Let $\w_m=\w_m(k)$ as in the proof of
Proposition \ref{weak-nonminimal-dense minimal}, and denote $\w_{m,j}=\w_m(k_j)$ for $1\le j\le m$.
That is, $\w_{m,j}$ is a word which contains all the $m$-name in $\{1,2,\ldots,k_j\}^m$.

For each $m\in \N$, choose columns $\mathfrak{c}_{n_m^1},\ldots, \mathfrak{c}_{n_m^{m(m+1)}}$
such that $$\mu(\bigcup_{i=1}^{m(m+1)}|\mathfrak{c}_{n_m^i}|)< \ep_m .$$
The columns should be disjoint from the columns in Step $k$, $k<m$. We assume that for each
$1\leq j\leq m$, $h_{n_m^{j(m+1)}}>\ldots >h_{n_m^{(j-1)(m+1)+1}}>2m^2k_j^{2m}$. For each
$1\leq j\leq m$, $1\leq i\leq m$, $s=(m+1)(j-1)+i$, let

\begin{equation*}
\begin{split}
   \xi_{n_m^{s}} & =\underbrace{\w_{2i-1,j}\w_{2i-1,j}\ldots\w_{2i-1,j} }_{m-j\
\text{times}}1^{h_{n_m^s}-(m-j)(2i-1)k_j^{2i-1}}\\ & =(\w_{2i-1,j})^{m-j}1^{h_{n_m^s}-(m-j)(2i-1)k_j^{2i-1}}
\in \{1,\ldots,k_j\}^{h_{n_m^{s}}}.
\end{split}
\end{equation*}
And let
$$\xi_{n_m^{j(m+1)}}=\w_{2m,j}1^{h_{n_m^{j(m+1)}}-2mk_j^{2m}}\in \{1,\ldots,k_j\}^{h_{n_m^{j(m+1)}}}.$$
Now for $1\leq i \leq m(m+1)$, copy the name $\xi_{n_m^i}$ to the column $\mathfrak{c}_{n_m^i}$ and we
obtain a new partition $\gamma_m^m$ with $d(\b_m,\gamma_m^m)<\ep_m$.

Inductively, we could construct a sequence of partitions $\{\gamma_n^n\}_n$ with the property
that $d(\b_m,\gamma_m^m)<\ep_m$ for each $m\in\N$. Now we need to build the
 required partition $\{\gamma_n\}$ from $\{\gamma_n^n\}_n$. First we construct
 partitions $\{\gamma_k^n\}_{n\in \N, 1\le k\le n}$ via $\{\gamma_n^n\}_n$. Then $\gamma_k=\lim_n \gamma_k^n$ is what we are looking for.

$$
\begin{array}{cccc}
    \gamma_1^1 &  &  &  \\
    \gamma_1^2 &  \gamma_2^2 &  &  \\
    \gamma_1^3 & \gamma_2^3 & \gamma_3^3 & \\
    \ldots & \ldots & \ldots & \ldots \\
    \downarrow & \downarrow &\downarrow &\downarrow\\
    \gamma_1 & \gamma_2 & \gamma_3 & \ldots
  \end{array}
$$

\medskip

Applying Lemma \ref{Fact} to $\beta_2,\gamma_2^2,$ and $\gamma_1^1$ we obtain $\gamma_1^2$. Similarly, applying Lemma \ref{Fact}
to $\beta_3,\gamma_3^3,$ and $\gamma_2^2$ we obtain $\gamma_2^3$, and we get $\gamma_1^3$ by applying Lemma \ref{Fact} to
$\beta_3, \gamma_3^3$ and $\gamma_1^2$. Inductively, we construct $\gamma_k^n$ for $k<n$ by applying Lemma \ref{Fact} and $\beta_n\prec \gamma_k^{n-1}$, for $k<n$.
Since $\b_1\prec \b_2\prec \ldots \prec \b_n$, we have
$\gamma_1^n\prec \gamma_2^n\prec\ldots \prec \gamma_n^n$ accordingly.

Since $d(\gamma_n^n,\b_n)<\ep_n$ and $\b_n\succ \gamma_{n-1}^{n-1}$, we know
that for each $1\le k\le n-1$, we have $d(\gamma_k^n,\gamma_k^{n-1})<\ep_n$.
That means for each $k$, $\{\gamma_k^n\}_{n\ge k}$ is a Cauchy sequence. So there is a partition $\gamma_k$ such that $\gamma_k^n\rightarrow \gamma_k$, as $n\rightarrow\infty$. Let $X_k^n$ denote the corresponding symbolic system of $\gamma_k^n$. The array shows the induction.

\[\begin{array}{cccc}
    X_1^1 &  &  &  \\
    X_1^2 &  X_2^2 &  &  \\
    X_1^3 & X_2^3 & X_3^3 & \\
    ... & ... & ... & ...
  \end{array}
\]

\medskip

By the construction, for each $k$, the sequence $\{\gamma_k^n\}_{n\ge k}$ has the same
property as $\{\a_n\}$ in Proposition \ref{weak-nonminimal-dense minimal}. According
to the proof of Proposition \ref{weak-nonminimal-dense minimal}, the corresponding
symbolic system $X_k=X_{\gamma_k}$ of $\gamma_k$ is  non-minimal topologically weakly
mixing with a dense set of minimal points.

Since for each $n\in \N$, $\gamma_1^n\prec \gamma_2^n\prec\ldots \prec \gamma_n^n$, we conclude that
$\{\gamma_k\}_k$ is increasing. As $\sigma(\b_k)\nearrow \mathcal{X}$, and
$d(\gamma_k,\b_k)<\sum_{s= k}^\infty \ep_s$, we deduce $\sigma(\gamma_k)\nearrow \mathcal{X}$ too.
\end{proof}

\medskip

Now Theorem \ref{main-result}(1) follows from Proposition \ref{increasing weak-nonminimal-dense minimal} and the following lemma.

\begin{lem}\label{inverse limit}
Let $(X,T)$ be the inverse limit of $\{(X_n,T_n)\}_n$, where each $(X_n,T_n)$ is a non-minimal
topologically weakly mixing system with a dense set of minimal points.
Then $(X,T)$ is also a non-minimal topologically weakly mixing system with a dense set of minimal points.
\end{lem}

\begin{proof}
By the definition of the inverse limit, it is easy to see that $(X,T)$ is not minimal as the factor of a minimal system
is minimal.

To show the density of minimal points in $X$ assume $U$
is a nonempty open set. Let $\pi_n:X\lra X_n$ be the projection. Then by the topology of $X$, there are
$n\in \N$ and an open non-empty set $U_n\subset X_n$ such that $\pi_n^{-1}U_n\subset U$. Let $x_n\in U_n$ be a minimal point
and $A$ be its orbit closure. Then there is a minimal set $B$ of $X$ such that $\pi_n(B)=A$. This implies that
there is a minimal point $x$ of $X$ such that $\pi_n(x)=x_n$ which implies that $x\in \pi_n^{-1}(x_n)\subset \pi_n^{-1}(U_n)\subset U$, and
hence the set of minimal point of $X$ is dense. The similar argument can be applied to show that
$(X,T)$ is weakly mixing. The proof of is completed.
\end{proof}

\section{Proof of Theorem \ref{main-result}-(2)}

In this section, we will prove Theorem \ref{main-result}-(2). First we will construct a model which is weakly mixing with a full support but its set of minimal points is not dense. Since in this case the closure of the set of minimal points has measure zero, we collapse it to a point and get the system required.

\medskip

First we need the following lemma (see \cite{Glasner,GW06,Weiss00} for a proof).

\begin{lem}\label{KR tower}
Let X be a non-periodic ergodic system.
For any positive integers $N_1,N_2$ with \((N_1,N_2)=1\), there exists a set $C$ of finite height such that the K-R tower \(\textbf{t}(C)\) satisfies range \(r_C\subset\{N_1,N_2\}\).
\end{lem}

To show Theorem \ref{main-result}-(2) we start with the following proposition and then follow the
standard procedure to finish the proof. To control the thickly syndetic sets, the construction here
is more involved than that in Proposition \ref{weak-nonminimal-dense minimal}.

\begin{prop}\label{weak-nonminimal-dense nominimal}
Let $(X,\X,\mu,T)$ be a non-periodic ergodic system and $\hat{\a}$ a finite partition of $X$.
Then for each $\ep>0$, there is a partition $\a$ such that the corresponding symbolic representation
$(X_\a,\X_\a,\rho,S)$ is a weakly mixing system whose set of minimal points is not dense, and
$$d(\a,\hat{\a})<\ep.$$
\end{prop}

\begin{proof}
The proof will be conducted by an inductive procedure. We first choose a sequence of positive real numbers
$\{\ep_n\}_{n=0}^\infty$ with $\sum_{n=0}^\infty\ep_n<\ep$. Then we  start from $\a_{-1}=\hat{\a}$
and construct $\{\a_n\}$ so that $d(\a_n,\a_{n+1})< \ep_{n+1}$ for $n\ge -1$. It is easy to see that the limiting
partition $\a$ satisfies $d(\hat{\a},\a)<\ep$.
To do so let $\a=\{A_1,A_2,\ldots,A_k\}$. On one hand, $\a$ is
constructed so that $(X_\a,T)$ is topologically weakly mixing. On the other hand, almost every point
will enter $A_2$ thickly syndetically so that the set of minimal points is not dense.
Now we begin our construction.
\medskip

\noindent {\bf Step $0$:}
Let $\hat{\a}=\{\hat{A}_1,\ldots,\hat{A}_k\}$. Let $\w_0$ be the name containing all pairs of names
of non-trivial elements in $\bigvee_{i=0}^1T^{-i}\hat{\a}$, where ``non-trivial elements'' in this
proof means the elements with positive measures.

Let $M_0=\min \{\mu(B): B\in \bigvee_{i=0}^1T^{-i}\hat{\a}\}$ and $0<\ep_0<\min\{\frac{\ep}{3},\frac{M_0}{3}\}$.
Choose $l_0\in \N$ such that $l_0>\max\{\frac{6}{\ep_0},2k^2\}$. Now for a fixed $N_0>\max\{\frac{6l_0}{\ep_0},\frac{6}{M_0}\}$,
by Lemma \ref{KR tower} there is a tower $\mathfrak{t}(C_0)=\{\mathfrak{c}_0^1,\mathfrak{c}_0^2\}$
such that heights of columns $\mathfrak{c}_0^1,\mathfrak{c}_0^2$ are $N_0,N_0+1$ respectively and the
corresponding bases are $C_0^1,C_0^2$. It is clear $C_0=C_0^1\cup C_0^2$. Put $e_0=\mu(C_0)$.
Copy the name $\w_0$ on the partial column $\{T^iC_0\}_{0\leq i\leq |\w_0|-1}$.
Then in column $\mathfrak{c}_0^1$, copy $2$ to the position $il_0$ for all $0\le i\le \frac{N_0-1}{l_0}$,
and in column $\mathfrak{c}_0^2$, copy $2$ to the position $il_0$ for all $0\le i\le \frac{N_0}{l_0}$.

In such a way we have constructed a new partition $\a_0 = \{A_1^0, A_2^0, \ldots, A_k^0 \}$. Note that
$d(\a_0,\hat{\a})<\frac{\ep_0}{3}$, since the measure changed is less than
$$(2k^2+\frac{N_0}{l_0})\mu(C_0) < (2k^2+\frac{N_0}{l_0}) \frac{1}{N_0} < \frac{\ep_0}{6} + \frac{\ep_0}{6} = \frac{\ep_0}{3}.$$

Let $A_{i_1}^0, A_{i_2}^0, A_{i_3}^0, A_{i_4}^0 \in \a_0$. Assume that positions of 2-name $(i_1,i_3),(i_2,i_4)$
appearing in $\w_0$ are $s$ and $t$. Then $T^sC_0\subseteq A^1_{i_1}$, $T^{s+1}C_0\subseteq A^1_{i_3}$, $T^r C_0
\subseteq A^0_{i_2}$ and $T^{r+1}C_0\subseteq A^0_{i_4}$. Hence
$$T^sC_0\times T^{r}C_0\subset (A_{i_1}^0\cap T^{-1}A_{i_3}^0)\times (A_{i_2}^0\cap T^{-1}A_{i_4}^0)
=(A_{i_1}^0\times A_{i_2}^0)\cap (T\times T)^{-1}(A_{i_3}^0\times A_{i_4}^0).$$
In particular, $$\mu\times \mu\big((A_{i_1}^0\times A_{i_2}^0)\cap (T\times T)^{-1}(A_{i_3}^0
\times A_{i_4}^0)\big)\ge \mu(C_0)^2=e_0^2>0.$$

\medskip

Now assume that inductively we have constructed partitions $\{\a_i\}_{i=0}^n$, two sequences of positive
integers $\{l_i\}_{0\leq i\leq n}$,$\{s_i\}_{0\leq i\leq n}$, two sequences of positive
numbers $\{\ep_i\}_{0\leq i\leq n},\{e_i\}_{0\leq i\leq n}$, with $\ep_{i+1}<\min\{\frac{\ep_i}{3},\frac{e_i^2}{3}\}$.
Also assume that we have obtained a sequence of K-R towers with relatively prime heights
$\{\mathfrak{t}(C_i^{(j)})\}_{1\leq i\le j\leq n}$ such that $C_0\supset C_1^{(1)}\supset
\ldots\supset C_n^{(n)}$, and the height of $\mathfrak{t}(C_i^{(n)})$ ranges in $[N_i,N_i+6N_{i-1}]$
with some positive integers $\{N_j\}_{1\leq j\leq n}$.

Let $\a_i=\{A_1^i,A_2^i,\ldots,A_k^i\}$ for $1\le \a_i \le n $.
The sequence $\{\a_i\}_{i=1}^n$ satisfies the following properties: for each $i\le n$

\begin{description}

  \item[$(1)_i$] We have $d(\a_{i-1},\a_{i})<\ep_{i}$. Let $\bigvee_{j=0}^{i}T^{-j}
  \a_{i-1}=\{U_1,\ldots,U_{\eta}\}$ with $U_j$ being nontrivial. Then there is a subset $\{B_1,\ldots,B_{\eta}\}
  \subset\bigvee_{j=0}^{i}T^{-j} \a_i$ such that the $\a_{i-1}$-name of $U_h$ and the $\a_i$-name of $B_h$ are
  the same for each $1\le h\le \eta$. Moreover, for all $E_1,F_1,E_2,F_2 \in \{U_1,\ldots, U_{\eta}\}$,
  one has that $$\mu\times\mu((T\times T)^{s_i}(E_1\times F_1)\cap(E_2\times F_2))>e_i^2>0.$$
      In particular, for all $E_1,F_1,E_2,F_2 \in \bigvee_{j=0}^{i-1}T^{-j} \a_{i-1}$, one has
      that $$\mu\times\mu((T\times T)^{s_i}(E_1\times F_1)\cap(E_2\times F_2))>e_i^2>0.$$

  \item[$(2)_i$] $C_i^{(i)}\subset C_{i-1}^{(i)}\subset \ldots \subset C_1^{(i)}$ and for $j\le i-1$,
  $\mu(|\mathfrak{t}(C_{j}^{(i-1)})|\Delta |\mathfrak{t}(C_j^{(i)})|)<\ep$.
  Refine the towers $\mathfrak{t}(C_j^{(i)})$ according to $\a_i$ for each $1\le j\le i$.
  For each $1\le j \le i$, if a column $\mathfrak{c}$ with base $C$ in the resulting tower
  $\mathfrak{t}(C_j^{(i)})$ has the $\a_i$-name $(a_1,a_2,\ldots,a_h)\in \Sigma^h$, then the name satisfies
      \begin{equation}\label{2i}
        a_{sl_j+t}=2 \ \text{for each}\ 0\leq t\leq j, 1\leq s\leq \tfrac{h-j-2}{l_j},
\end{equation}
 i.e. $T^{sl_j+t}C\subset A_2^i$.

\medskip

\[\begin{array}{cccccc}
   \mathfrak{t}(C_1^{(1)}) & \mathfrak{t}(C_1^{(2)}) & \mathfrak{t}(C_1^{(3)}) & \cdots & \rightarrow & \mathfrak{t}(C_1^{*}) \\
    & \mathfrak{t}(C_2^{(2)}) & \mathfrak{t}(C_2^{(3)})  & \cdots &\rightarrow &  \mathfrak{t}(C_2^{*})  \\
    &  & \mathfrak{t}(C_3^{(3)}) & \cdots &\rightarrow  &\mathfrak{t}(C_3^{*}) \\
    &  &  &  \cdots & \cdots &\cdots
  \end{array}
\]

\medskip
\end{description}

\medskip

Note that $(1)_{i}$ will be used to show that $X_\a$ is weakly mixing, and $(2)_{i}$ will be used to show that the minimal points are not dense in $X_\a$.

\medskip

\noindent {\bf Step $n+1$:} Now we make the induction for the $n+1$ case. First we need to define a word $\w_{n+1}$
which contains all pairs of names of non-trivial elements in $\bigvee_{i=0}^{n+1}T^{-i}\a_n$. We do it as follows.

Refine the tower $\mathfrak{t}(C_n^{(n)})$ according to $\a_n$, and let the resulting tower be
$\mathfrak{t}(C_n^{(n)})=\{\mathfrak{c}_n^j\}_j$. Note that the height of each column is in
$[N_n,N_n+6N_{n-1}]$. Let $W_{n+2}=\{B_1,B_2,\ldots,B_t\}\subset \{1,\ldots,k\}^{n+2}$ be the
set of all names of nontrivial elements of $\bigvee_{i=0}^{n+1}T^{-i}\a_n$.
Each $(n+2)$-word $B_j$ ($j\in \{1,\ldots,t\}$) in $W_{n+2}$ either appears in some
column $\mathfrak{c}_n^{i_j}$ of $\mathfrak{t}(C_n^{(n)})$, or appears in the concatenation
of two columns of $\mathfrak{t}(C_n^{(n)})$ (i.e. there are $\mathfrak{c}_a,\mathfrak{c}_b$
in $\mathfrak{t}(C_n^{(n)})$ such that the name appears in $\mathfrak{c}_a\mathfrak{c}_b$).
In the second case we also use $\mathfrak{c}_n^{i_j}$ to denote the  concatenation of two
columns. Let $\tilde{B}_j$ be the name of $\mathfrak{c}_n^{i_j}$.

Now fix a large number $s_{n+1}>10 N_n^2$, and construct the word $\w_{n+1}$ as follows:
For each pair $(j_1,j_2)\in \{1,\ldots,t\}^2$, make sure that words $\mathfrak{c}_n^{i_{j_1}}$
and $\mathfrak{c}_n^{i_{j_2}}$ appear in $\w_{n+1}$, and the distance from the word $B_{j_1}$ to the
word $B_{j_2}$ is $s_{n+1}$. Since the column heights of $\mathfrak{t}(C_n^{(n)})$ are relatively
prime and $s_{n+1}$ is large enough, one can use $\a_n$-names of columns $\{\mathfrak{c}_n^i\}_i$
to fill gaps between each pair $\tilde{B}_i,\tilde{B}_j$.

\medskip

Let $M_{n+1}<\frac{1}{2}\min_{B\in \bigvee_{i=0}^{n+2}T^{-i}{\a_n}}\{\mu(B)\}$, $\ep_{n+1}
<\min\{\frac{\ep_n}{3},\frac{e_n^2}{3}\}$. Then let $l_{n+1}>\max\{|\w_{n+1}|+10N_n^2+3N_n,
\frac{6n}{\ep_{n+1}}\}$ and $N_{n+1}>\max\{\frac{6(n+1)l_{n+1}}{\ep_{n+1}},\frac{n+3}{M_{n+1}}\}$.
By Proposition \ref{KR extra prime tower}, we have a new K-R tower $\mathfrak{t}(C_{n+1}^{(n+1)})$
with relatively prime column heights and $C_{n+1}^{(n+1)}\subset C_n^{(n)}$, and its height ranges
in $[N_{n+1},N_{n+1}+6N_n]$. Refine $\mathfrak{t}(C_{n+1}^{(n+1)})$ according to $\a_n$, and let
the resulting tower be $\{\mathfrak{c}_{n+1}^j\}_j$. Let the base of $\mathfrak{c}_{n+1}^j$ be $C_{n+1}^j$,
and let its height be $H_j$. Let $e_{n+1}=\min_i\{\mu(C_{n+1}^i)\}$. Now we do the following adjustment
for each column $\mathfrak{c}_{n+1}^j$.

Denote the name $\mathfrak{c}_{n+1}^j$ by $(c_1,c_2,\ldots,c_{H_j})\in \Sigma^{H_j}$.
First, copy the name $\w_{n+1}$ to $(c_h, \ldots, c_{|\w_{n+1}|+h-1})$, where $h>n+3$ is the
first number such that $T^{h-1}C_{n+1}^j\subset C_n^{(n)}$.
Secondly, we choose a $R\in \N$ such that
$$l_{n+1}-2N_n\leq R<l_{n+1}, R-(|\w_{n+1}|+h-1)>10N_n^2,\ \text{and}\ T^{R-1}C_{n+1}^j\subset C_n^{(n)}.$$
Since the column heights of the tower $\mathfrak{t}(C_n^{(n)})$
are relatively prime, we can replace $(c_{|\w_{n+1}|+1},\ldots,c_{R-1})$ by the names encountered
in the tower $\mathfrak{t}(C_n^{(n)})$. 

Finally, copy $2$ to $c_{sl_{n+1}+r}$ for each $0\leq r\leq n+1,
1\leq s\leq \frac{H-n-2}{l_{n+1}}$. Then according to the new name we have a new partition $\a_{n+1}$.

\medskip
\noindent{\bf Properties of $\alpha_{n+1}$}:
Note that by the construction of $\a_{n+1}$ if we refine the tower $\mathfrak{t}(C_{n+1}^{(n+1)})$
according to $\a_{n+1}$, then the resulting tower will still be $\{\mathfrak{c}_{n+1}^j\}_j$. Since
column heights of the tower $\mathfrak{t}(C_n^{(n)})$ are relatively prime, we have made sure that
the first $l_{n+1}$ length part of the name along the column in $\{\mathfrak{c}_{n+1}^j\}_j$ consists
only of the name encountered in the tower $\mathfrak{t}(C_n^{(n)})$. These change the levels where the
bases of the $\mathfrak{t}(C_n^{(n)})$ name blocks occur. Thus it defines a new base which we called
$C_n^{(n+1)}$, and therefore a new K-R tower $\mathfrak{t}(C_n^{(n+1)})$. Since $d(\a_n,\a_{n+1})<\ep_{n+1}$,
changes from the tower $\mathfrak{t}(C_n^{(n)})$ to $\mathfrak{t}(C_n^{(n+1)})$ are very small (less than $\ep_{n+1}$).
Since we copy $2$ to $c_{sl_{n+1}+r}$ for each $0\leq r\leq n+1,1\leq s\leq \frac{H-n-2}{l_{n+1}}$,
each $\a_{n+1}$-name of $\mathfrak{t}(C_n^{(n+1)})$ either has the same name with some column
in $\mathfrak{t}(C_n^{(n)})$, or has more $2$ appeared than  some column name in $\mathfrak{t}(C_n^{(n)})$.
Anyway, for each name with the length $h$ in $\mathfrak{t}(C_n^{(n+1)})$, in the positions
$sl_n+t, \forall 0\le t\le j, 1\le s\le \frac{h-n-2}{l_n}$ the names are $2$.

By $(2)_n$, $C_n^{(n)}\subset C_{n-1}^{(n)}\subset \ldots \subset C_1^{(n)}$, above changes from
the tower $\mathfrak{t}(C_n^{(n)})$ to the tower $\mathfrak{t}(C_n^{(n+1)})$ will induce corresponding
changes such that the tower $\mathfrak{t}(C_j^{(n)})$ will become some new tower $\mathfrak{t}(C_j^{(n+1)})$
for each $1\le j\le n-1$, where $C_{n+1}^{(n+1)}\subset C_{n}^{(n+1)}\subset \ldots \subset C_1^{(n+1)}$.
By the same reason as showed for $\mathfrak{t}(C_n^{(n+1)})$, equality (\ref{2i}) holds for each $j\le n+1$.
Thus we have $(2)_{n+1}$.

\medskip

Now we verify that $\a_{n+1}$ satisfies $(1)_{n+1}$.

\medskip

By the construction, the measure changed from $\a_n$ to $\a_{n+1}$ is less than
\begin{equation*}
\begin{split}
& \quad \ \mu(C_{n+1})(|\w_{n+1}|+(n+2)\tfrac{N_{n+1}+6N_n}{l_{n+1}})\\ & < \frac{1}{N_{n+1}}(l_{n+1}+(n+2)\frac{2N_{n+1}}{l_{n+1}})\\ & < \frac{1}{6n+1}\ep_{n+1}+\frac{2(n+1)}{l_{n+1}}<\frac{\ep_{n+1}}{2}+\frac{\ep_{n+1}}{2}=\ep_{n+1}.
\end{split}
\end{equation*}
Thus we conclude that $d(\a_n,\a_{n+1})<\ep_{n+1}$. And the second part of $(1)_{n+1}$ is guaranteed by the construction of $\w_{n+1}$.

Let $D_{i_1},D_{i_2},D_{j_1},D_{j_2}\in \bigvee_{i=0}^{n+1}T^{-i}\a_{n}$, and let their names
 be $B_{i_1}, B_{i_2}, B_{j_1}, B_{j_2}\in W_{n+2}$ respectively, where $1\le i_1,i_2,j_1,j_2\leq t$.
 Then by the definition of $\w_{n+1}$, pairs $(B_{i_1}, B_{j_1})$ and $(B_{i_2}, B_{j_2})$ appear
 in the word $\w_{n+1}$. Given arbitrary column $\mathfrak{c}_{n+1}^i$ with the base $C_{n+1}^i$,
 let $p$ be the position of $B_{i_1}$ in this column and let $r$ be the distance from the
 position of $B_{i_1}$ to the position of $B_{i_2}$. Then we have:
$$T^{p-1}C_{n+1}^i\subset D_{i_1}, T^{p-1+s_{n+1}}C_{n+1}^i\subset D_{j_1},T^{p-1+r}
C_{n+1}^i\subset D_{i_2},T^{p-1+r+s_{n+1}}C_{n+1}^i\subset D_{j_2}.$$
It follows that
\begin{equation*}
\begin{split}
T^{p-1}C_{n+1}^i\times T^{p-1+r}C_{n+1}^i&\subset (D_{i_1}\cap T^{-s_{n+1}}D_{j_1})\times
(D_{i_2}\cap T^{-s_{n+1}}D_{j_2})\\ & =(D_{i_1}\times  D_{i_2})\cap(T\times T)^{-s_{n+1}}( D_{j_1}\times  D_{j_2})
\end{split}
\end{equation*}
Hence
\begin{equation*}
\begin{split}
& \quad \ \mu\times\mu((D_{i_1}\times  D_{i_2})\cap(T\times T)^{-s_{n+1}}( D_{j_1}\times
D_{j_2}))\\ & \geq \mu\times\mu(T^{p-1}C_{n+1}^i\times T^{p-1+r}C_{n+1}^i)\ge e_{n+1}^2>0.
\end{split}
\end{equation*}
Thus $(1)_{n+1}$ holds.


\medskip
\noindent{\bf Properties of $\alpha$}:
So by the induction we have a sequence of partitions $\{\a_n\}$ and assume that limit
partition is $\a=\{A_1,A_2,\ldots,A_k\}$. It is clear $$d(\hat{\a},\a)<\sum_{i=0}^{\infty}\ep_i<\ep.$$
Also by the condition $(2)_n$, for $n\ge 1$ each sequence $\{\mathfrak{t}(C_n^{(j)})\}_{j\ge n}$
has a limit tower $\mathfrak{t}(C_n^*)$ with base $C_n^*$. And by $(2)_n$, $C_1^{*}\supset C_{2}^{*}\supset \ldots $.

\medskip

Now we show $\a$ is the partition  required. First we claim that $\a$ satisfies the following properties:

\begin{enumerate}
   \item For each $m\geq 0$, $E_1,F_1,E_2,F_2\in \bigvee_{j=0}^{m-1}T^{-j} \a$, we have that
$$\mu\times\mu((T\times T)^{s_m}(E_1\times F_1)\cap(E_2\times F_2)>0.$$

  \item Refine the towers $\mathfrak{t}(C_j^{*})$ according to $\a$ for each $j\ge 1$. If column $\mathfrak{c}$
  in the resulting tower $\mathfrak{t}(C_j^{*})$ has the $\a$-name $(a_1,a_2,\ldots,a_h)\in \Sigma^h$ and let its base be $C$, then the name satisfies
      \begin{equation}\label{2}
        a_{sl_j+t}=2 \ \text{for each}\ 0\leq t\leq j, 1\leq s\leq \tfrac{h-j-2}{l_j},
\end{equation}
 i.e. $T^{sl_j+t}C\subset A_2$.
\end{enumerate}

\medskip

Condition (2) is guaranteed by $(2)_n$. It is left to verify the condition (1). By condition $(1)_m$
there are $E'_1,F'_1,E'_2,F'_2\in \bigvee_{j=0}^{m-1}T^{-j} \a_{m-1}$ such that they have the same
names with $E_1,F_1,E_2,F_2$ respectively. By $(1)_{m}$
$$\mu\times\mu((T\times T)^{s_m}(E'_1\times F'_1)\cap(E'_2\times F'_2))>e_{m}^2.$$
Then by $d(\a_m,\a)<\sum_{j=m+1}^\infty \ep_j$, one has that
$$\mu\times\mu((T\times T)^{s_m}(E_1\times F_1)\cap(E_2\times F_2)>e_{m}^2-\sum_{j=m+1}^{\infty}\ep_j>0.$$

\medskip

Now using conditions (1) and (2) we will show $\a$ is what we need.
Let $X_{\a}$ be the corresponding symbolic representation of $\a$, and $\phi: X\rightarrow X_\a$ be
the factor map. Let $[i]_0=\{w\in X_\a: w_0=i\}$ for $i\in \{1,2,\ldots, k\}$.
Let $w=\phi(x)\in [1]_0$ be a transitive point of $(X_\a, T)$.

By property (1), $(X_\a,T)$ is weakly mixing. By property (2), $N(w, [2]_0)$ is
thickly syndetic, which implies that $N(w,[1]_0)$ is not piecewise syndetic. Hence the set of minimal points of $(X_\a,T)$ is not dense.
\end{proof}

Similar to Lemma \ref{inverse limit} we have the following easy observation.

\begin{lem}\label{inverse limit2}
Let $(X,T)$ be the inverse limit of $\{(X_n,T_n)\}_n$, where each $(X_n,T_n)$ is a
non-minimal topologically weakly mixing system whose set of minimal points is not dense.
Then $(X,T)$ is also a non-minimal topologically weakly mixing system whose set of minimal points is not dense.
\end{lem}

Using the similar argument that we obtain Proposition \ref{increasing weak-nonminimal-dense minimal} from Proposition
\ref{weak-nonminimal-dense minimal}, and adjusting the proof of Proposition \ref{weak-nonminimal-dense nominimal},
we deduce the following result.

\begin{prop}\label{model for 2}
Every non-periodic ergodic system has a topological model which is a weakly mixing system with a full support and the set of minimal points is not dense.
\end{prop}

\begin{proof}
The idea of the proof is similar to the one used in the proof of Proposition~ \ref{increasing weak-nonminimal-dense minimal}.
We will show that there exists an increasing sequence of finite partitions $\{\gamma_n\}$ such that
$\sigma(\gamma_n)\nearrow \X$ and for each $n\in \N$ the corresponding symbolic representation $(X_{\gamma_n},
\X_{\gamma_n},\rho_n,S)$ is a weakly mixing system with a full support and the set of minimal points is
not dense. Then by Lemma \ref{inverse limit2}, we finish the proof.

Let $(X,\X,\mu, T)$ be the ergodic system. Let $\{\b_n\}_{n\ge 0}$ be an increasing sequence of finite
partitions such that $\sigma(\b_n)\nearrow \X$. And let $\{\ep_n\}$ be a sequence of positive
numbers with $\sum_{n=0}^{\infty} \ep_n<\infty$.
We will modify the proof of Proposition \ref{weak-nonminimal-dense nominimal} carefully to get what we need.

As in the proof of Proposition \ref{weak-nonminimal-dense nominimal} we choose a
tower $\mathfrak{t}(C_0^{0})$, and adjust $\b_0$ by Step~ $0$ to get a new partition $\gamma_0^0$.
We replace $\b_1$ by $\b_1\bigvee \gamma_0^0$  (still denote it by $\b_1$), and it is clear $\gamma_{0}^{0}\prec \b_1$.
We assume that the first element (resp. second element) of $\b_1$ is a subset of the first element (resp. the second element) of $\gamma_{0}^{0}.$

As in Step 1 of the proof of Proposition \ref{weak-nonminimal-dense nominimal},
we modify $\b_1$ to deduce a new partition $\gamma_1^1$. We then construct a tower $\mathfrak{t}(C_1^{1})$
using Lemma \ref{KR extra prime tower}, and form a new tower $\mathfrak{t}(C_0^{1})$.

By Fact in the proof of Proposition \ref{increasing weak-nonminimal-dense minimal}, we
construct $\gamma_0^1\prec \gamma_1^1$. Refining $\gamma_0^1$ to $\mathfrak{t}(C_0^{1})$,
we know that $\gamma_0^1$ satisfying $(1)_1,(2)_1$ in Step 1 of the proof of
Proposition \ref{weak-nonminimal-dense nominimal} since $\b_1\succ \gamma_0^0$.

Inductively, we replace $\b_n$ by $\b_n\bigvee \gamma_{n-1}^{n-1}$. And we assume that
the first element (resp. second element) of $\b_n$ is a subset of the first element (resp. second element) of $\gamma_{n-1}^{n-1}$.

We modify $\b_n$ by Step $n$ to get a new partition $\gamma_n^n$ such that $d(\b_n, \gamma_n^n)<\ep_{n}$,
and by the same argument we know that $\gamma_n^n$ satisfies the same properties listed in $(1)_{n}$ and $(2)_n$ for the
tower $\mathfrak{t}(C_n^{(n)})$. Now construct $\gamma_k^n\prec \gamma_n^n$ by Lemma \ref{Fact}. Since the first and second elements of
$\b_n$ are subsets of the first and second elements of $\gamma_{n-1}^{n-1}$ respectively,
and $\b_n\succ \gamma_{n-1}^{n-1}$, we conclude that $\gamma_k^n$ satisfies the same
properties listed in $(1)_{n}$ and $(2)_n$  for the tower $\mathfrak{t}(C_j^{(n)})$, $k\leq j\leq n$.
By the proof of Proposition \ref{weak-nonminimal-dense nominimal}, the partition $\gamma_k=\lim_n\gamma_k^n$
satisfies properties as $(1),(2)$ there. Hence according to the proof of Proposition \ref{weak-nonminimal-dense nominimal},
$X_{\gamma_k}$ is a weakly mixing system with a full support and the set of minimal points is not dense.

Following the same discussion as in the proof of Proposition \ref{increasing weak-nonminimal-dense minimal},
we know that $\{\gamma_k\}$ is increasing and $\sigma(\gamma_k)\nearrow \X$. The proof is completed.
\end{proof}

\medskip

Now using Proposition \ref{model for 2}, we are able to finish the proof of Theorem \ref{main-result}-(2).

\begin{proof}[Proof of Theorem \ref{main-result}-(2)]
For a given ergodic system $(X,\X,\mu,T)$, by Proposition~ \ref{model for 2}, there is a topological model
$(Y,S)$ of $X$ with an ergodic measure $\rho$, which is weakly mixing, non-minimal and the set of
minimal point ${\rm Min} (Y)$ is not dense in $supp(\rho)=Y$. Note that $\rho(\overline{{\rm Min}( Y)})=0$,
since $\overline{{\rm Min}(Y)}$ is an $S$-invariant set.

Define an equivalence relation $'\thicksim'$ in $Y$ as follows:
$x \thicksim y$ if $x,y\in \overline{{\rm Min}(Y)}$.
Then the quotient system $(\h{X}=Y/\thicksim, \h{T})$ is a system that is measure theoretically
isomorphism to $(Y,S)$ since $\rho(\overline{{\rm Min}( Y)})=0$. Hence $(\h{X},\h{T})$ is also a
topological model of $(X,\X,\mu,T)$. Note that $(\h{X},\h{T})$ is a topologically weakly mixing
system with a full support and a unique fixed point as its only minimal point. Thus the proof is completed.
\end{proof}

\section{Applications}

In this section we give two applications of the results we obtained. Let $(X,T)$ be a topological dynamics
and $M(X)$ is the collection of all Borel probability measures on $X$ with the weak$^*$ topology. Then
$T$ induces a map $T_M$ on $M(X)$ naturally by sending $\mu\in M(X)$ to $T\mu$. An unsolved question in \cite{LYY13}
is that if there is a weakly mixing proximal system $(X,T)$ such that $(M(X),T_M)$ has dense minimal points.
We give an affirmative answer to this question. That is,

\begin{thm}\label{Appli1}
There is a weakly mixing proximal system $(X,T)$ such that $(M(X),T_M)$ has dense minimal points.
\end{thm}

To show this result we need a lemma from \cite{LYY13}.

\begin{lem} \label{lem:2-2}
Let $X, Y$ be two compact metric spaces, $\mu \in M(X)$ and $\nu \in
M(Y)$.

\begin{enumerate}
  \item  If $A=\bigcup_{i=1}^n A_i$, where $A_1, \ldots, A_n$ are Borel subsets of $X$ with $\mu(A_i)>0$ and $\mu(A_i \cap A_j)=0$
  for all $1 \le i<j \le n$, then $\mu_A=\sum_{i=1}^n \frac{\mu(A_i)}{\mu(A)} \mu_{A_i}$.
  \item Let $\epsilon>0$ and $A$ be a Borel subset of $X$ with $\mu(A)>0$. If $B$ is a Borel subset of $X$ such that
  $\mu(B)>0$ and $\mu(A \D B)< \mu(A) \cdot \epsilon$, then $d(\mu_A, \mu_B) \le 2\epsilon$.

  \item If $\pi: (X, \mu) \rightarrow (Y, \nu)$ is measurable and $\pi \mu=\nu$, then $\pi \mu_{\pi^{-1}A}=\nu_A$ for each Borel subset $A$ of $Y$.
\end{enumerate}
\end{lem}

\medskip

\noindent {\it Proof of  Theorem \ref{Appli1}:} Let $(\Sigma_2, T)$ be the dyadic adding machine with a unique
ergodic measure $\mu$.  By Theorem \ref{main-result} $(\Sigma_2, T, \mu)$ is isomorphic to $(Y, S, \nu)$,
where $(Y, S)$ is a weakly mixing proximal topological system and $\nu$ has full support. We now
show  that the set of periodic points of $(M(Y), S_M)$ is dense.

Let $\pi: (\Sigma_2, T, \mu) \rightarrow (Y, S, \nu)$ be an
isomorphism, that is,  there are invariant Borel subsets $X_1
\subset X$ and $X_2 \subset Y$ with $\mu(X_1)=\nu(X_2)=1$ and an
invertible measure-preserving transformation $\pi: X_1 \rightarrow
X_2$ such that $\pi(Tx)=S\pi(x)$ for all $x \in X_1$.

Let $\epsilon>0$ and let $U$ be a non-empty open subset of $Y$.
Since $\nu$ has full support, we have $\nu(U)>0$.
Thus, there are finitely many pairwise disjoint cylinders $A_1,
\ldots, A_k$ of $X$ such that $\mu(\pi^{-1}U \D A)<
\nu(U) \cdot \epsilon$ with $A=\bigcup_{i=1}^k A_i$, which implies
$\nu(U \D \pi(A \cap X_1))<\nu(U) \cdot \epsilon$. Using
Lemma \ref{lem:2-2} (2), $d(\nu_U, \nu_{\pi(A \cap X_1)})\le
2\epsilon$. Since $T^{2^{|C|}}C=C$ for each cylinder $C$ of $X$,
where $|C|$ stand for the length of $C$, we conclude that $\mu_C$ is
periodic. In particular, each $\mu_{A_i}$ is periodic. By Lemma
\ref{lem:2-2} (3), each $\nu_{\pi(A_i \cap X_1)}$ is also periodic.
By Lemma \ref{lem:2-2} (1), $\nu_{\pi(A \cap X_1)}=\sum_{i=1}^k p_i
\nu_{\pi(A_i \cap X_1)}$, where $p_i=\mu(A_i)/\mu(A)$. Thus,
$\nu_{\pi(A \cap X_1)}$ is periodic. It follows that $\nu_U$ is
approached by periodic points of $(M(Y), S_M)$.

Now take $y \in Y$ and let $\{U_n\}_{n=1}^{\infty}$ be a sequence of
open neighborhoods  of $y$ such that $\mathrm{diam}(U_n) \rightarrow
0$. For any $f \in C(Y, \mathbb{R})$, we have
$$\left|\int_Y \, f(z) \, \mathrm{d}\nu_{U_n}-f(y)\right|\le \int_{U_n} \left|f(z)-f(y)\right| \mathrm{d}\nu_{U_n} \rightarrow 0.$$
A simple calculation  shows $\nu_{_{U_n}} \rightarrow \delta_y$, and
hence $\delta_y$ is a limit point  of $P(S_M)$. This implies that
each element of $M_n(Y)=\{\frac{1}{n}\sum_{i=1}^n \delta_{x_i}: x_i \in X \}$ is approached by elements of $P(S_M)$.
Since $\bigcup_{n=1}^{\infty} M_n(Y)$ is dense in $M(Y)$, it follows
that $(M(Y), S_M)$ is a $P$-system. This ends the proof.

\medskip

Another application of our result is the following.
A topological analogy of K-systems, called {\it topological K-system} was studied in \cite{HY06}.
In \cite{HLY12} the authors constructed a proximal topological K-system which is weakly mixing. Using Theorem \ref{main-result},
we can get a lot of such examples which are strongly mixing.

\begin{thm}
There exist strongly mixing proximal topological K-systems.
\end{thm}

\begin{proof}
Let $(X, T, \mu)$ be a measurable K-system. By Theorem \ref{main-result}
$(X, T, \mu)$ is isomorphic to a proximal system $(Y,S)$ with a measure $\nu$ of full support.
Thus $(Y,S)$ is strongly mixing, since a K-system is strongly mixing in the measurable sense.
At the same time we know that $(Y,S)$ is topological K by \cite[Theorem 3.4]{HY06}.
\end{proof}


\end{document}